\documentclass[12pt]{amsart}

\usepackage{amssymb,stmaryrd}
\usepackage{rotating}

\usepackage{mathtools}

\usepackage{enumerate}
\usepackage{color}

 \newcommand{\sll}{\mathfrak{so}(4,\mathbb{C}) }
 \newcommand{\s}{\mathfrak{sl}(2,\mathbb{C})}
 \newcommand{\ssl}{\mathfrak{sl}(2,\mathbb{C})}
  
 \newcommand{\la}{\mathfrak{g}}

\newcommand{\e}{\mathfrak{e}(2)}

\newcommand{\soo}{\mathfrak{sp}(4,\mathbb{C})}

  \theoremstyle{definition}
  \newtheorem{definition}{Definition}[section]

  \theoremstyle{plain}
  \newtheorem{lemma}[definition]{Lemma}
  
  \newtheorem{theorem}[definition]{Theorem}

\title[]{The  subalgebras of $\sll$}

\begin{document}

\author[]{Andrew Douglas}

\author[]{Joe Repka}

\address[]{CUNY Graduate Center and New York City College of Technology, City University of New York, USA}
\email{adouglas2@gc.cuny.edu}
\address[]{Department of Mathematics, University of Toronto, Canada}
\email{repka@math.toronto.edu}

\date{\today}


\begin{abstract}
We classify  the solvable subalgebras, semisimple subalgebras, and Levi
decomposable subalgebras of $\sll$, up to inner automorphism.  By Levi's Theorem, this is a full classification of the subalgebras of $\sll$. 
\end{abstract}

\maketitle

\section{Introduction}

Semisimple subalgebras of semisimple Lie algebras have been extensively studied.  Dynkin \cite{dynkin} and Minchenko \cite{min}, for instance, classified the semisimple subalgebras of the exceptional Lie algebras, up to inner automorphism.  In \cite{degraafb}, de Graaf classified the semisimple subalgebras of the simple Lie algebras of ranks $\le 8$, up to linear equivalence, which is somewhat weaker than a classification up to inner automorphism.

Much less research has examined general subalgebras of semisimple Lie
algebras.  By Levi's Theorem [\cite{levi}, Chapter III, Section 9] any
finite-dimensional Lie algebra over a field of characteristic $0$ is a
semi-direct sum of its maximal solvable ideal and a semisimple Lie
algebra.   A Lie algebra with a nontrivial decomposition into  a semisimple Lie
algebra with a solvable Lie algebra is referred to as a  Levi
decomposable  algebra.

We have made considerable progress in classifying both solvable and Levi decomposable subalgebras of semisimple Lie algebras (e.g., \cite{abcd, dr1, dkr, ddr}).   For instance, in \cite{dkr}, we classified the abelian extensions of $\mathfrak{so}(2n,\mathbb{C})$ in the exceptional Lie algebras $E_{n+1}$, up to inner automorphism.   We classified subalgebras isomorphic to the complexification of the  (solvable) Euclidean Lie algebra $\mathfrak{e}(2)$ in the rank $2$ classical Lie algebras in \cite{dr1}; and the subalgebras isomorphic to  the (Levi decomposable) Poincar\'e algebra 
in rank $3$ simple Lie algebras in \cite{ddr}.

More recently, in \cite{abcd} we classified    
the subalgebras isomorphic to  $\mathfrak{sl}(n,\mathbb{C})\inplus \mathbb{C}^{n+1}$, $\mathfrak{so}(2n+1,\mathbb{C})\inplus  \mathbb{C}^{2n+1}$, $\mathfrak{sp}(2n,\mathbb{C})\inplus  \mathbb{C}^{2n}$, $\mathfrak{so}(2n,\mathbb{C})\inplus  \mathbb{C}^{2n}$  in the simple Lie algebras $\mathfrak{sl}(n+1,\mathbb{C})$, $\mathfrak{so}(2n+3,\mathbb{C})$, $\mathfrak{sp}(2n+2,\mathbb{C})$, $\mathfrak{so}(2n+2,\mathbb{C})$, respectively, up to inner automorphism.

In the present article we extend the work on subalgebras of semisimple Lie algebras by classifying  the subalgebras of $\sll$, up to inner automorphism.   The classification of the semisimple
subalgebras of $\sll$ is straightforward.  We classify its solvable
and Levi decomposable subalgebras to complete the classification of 
the subalgebras of $\sll$.

The article is organized as follows.  In Section \ref{begin} we describe the semisimple Lie algebra $\sll$.  In Section \ref{add} we give 
additional notation and terminology.   In Section \ref{solvable} we present the classification of  solvable Lie algebras of  small dimension from \cite{degraafb}.   
We then classify the solvable subalgebras of $\sll$ in  
Section 
\ref{solvableb}.  Finally, in sections \ref{semi} and  \ref{les}, we describe the classification of the semisimple subalgebras of $\sll$ and classify its Levi decomposable subalgebras.

\section{The semisimple Lie algebra $\sll$ and its representations}\label{begin}

The special linear algebra $\s$ is the simple Lie algebra of traceless $2 \times 2$ matrices with complex entries.  
For each $m \in \mathbb{Z} _{\geq 0}$,  there is an irreducible representation $V(m)$ of $\s$ with $\dim (V(m)) = m$.  
Moreover,  these are all the irreducible representations of $\s$.  
The  semisimple Lie algebra $\sll$ is isomorphic to $\s \oplus \s$.  It  is of type $A_1 \times A_1$, and has a Chevalley basis $\{ x_i, y_i ,  h_j:  1\leq i \leq 2, 1 \leq j \leq 2 \}$  defined as follows:
\begin{equation}
\begin{array}{l}
\displaystyle  ah_1+ bh_2+  cx_1+ dx_2+  c'y_1+ d'y_2=  \left(
\begin{array}{rrrr}
 a & c & 0  &0 \\
 c' & -a & 0&  0 \\
 0& 0 &b  &d\\
 0&0&d'&-b
\end{array}
\right).
\end{array} 
\end{equation}

We end this section with a description of the  representations of $\sll$ $\cong$ $\ssl$ $\oplus$ $\ssl$.  Its representations are constructed from those of $\s$.  If $V_1$ and $V_2$ are $\s$-modules, then $V_1 \otimes V_2$ is an $\soo$-module with action
\begin{equation}
(L_1,L_2) \cdot (v_1 \otimes v_2)= (L_1 \cdot v_1) \otimes v_2  + v_1 \otimes (L_2 \cdot v_2).
\end{equation}
We have the following  well-known theorem  
 classifying the finite-dimensional, irreducible representations of $\ssl$ $\oplus$ $\ssl$:
The finite-dimensional $\ssl$ $\oplus$ $\ssl$ representation $V$ is irreducible  if and only if $V \cong V(n) \otimes V(m)$ for some $n$ 
and $m$ $\in$ $\mathbb{Z} _{\geq 0}$, uniquely determined by $V$.

\section{Additional notation and terminology}\label{add}

\begin{itemize} 
\item Two Lie algebra embeddings $\varphi$ and $\varphi'$ into $\la$ are equivalent (up to inner automorphism) if there is an inner automorphism $\psi$ of $\la$ such that $\psi \circ \varphi =\varphi'$.  In this case we write
$\varphi \sim \varphi'.$
\item  Two subalgebras $S$ and $S'$ of $\la$ are equivalent  if  there is an inner automorphism $\psi$ of $\la$ such that $\psi(S)=S'$.  In this case we write
$S \sim S'.$
\end{itemize}

\section{Classification of solvable Lie algebras of small dimension}\label{solvable}

A full classification of solvable Lie algebras is not known and thought to be an impossible task.  However, classifications of solvable Lie algebras in special cases have been considered.  For instance, de Graaf classified the solvable Lie algebras in dimension $\leq 4$ over fields of any characteristic \cite{degraafa}.  
Writing $\{ z_{1}, \dots , z_{m} \}$ for a basis of an algebra of dimension $m$,  we
 describe 
the
classification, which we consider over $\mathbb{C}$, below:
\begin{equation}
\begin{array}{llll}
J & \text{The abelian Lie algebra of dimension $1$}\\
\end{array}
\end{equation}
\begin{equation}
\begin{array}{llll}
K_1 & \text{The abelian Lie algebra of dimension $2$}\\
K_2 & [z_1,z_2]=z_1 \\
\end{array}
\end{equation}
\begin{equation}
\begin{array}{llll}
L_1 & \text{The abelian Lie algebra of dimension $3$}\\
L_2 & [z_3,z_1]=z_1, [z_3,z_2]=z_2 \\
L_{3,a} & [z_3,z_1]=z_2, [z_3,z_2]=az_1+z_2 \\
L_{4} & [z_3,z_1]=z_2, [z_3,z_2]=z_1 \\
L_{5} & [z_3,z_1]=z_2
\end{array}
\end{equation}
Note that we get a nonisomorphic solvable Lie algebra $L_{3,a}$ for each $a\in \mathbb{C}$.

There are numerous families of solvable Lie algebras of dimension $4$, some of which are infinite \cite{degraafa}.  We describe only one $4$-dimensional solvable algebra from this classification, which de Graaf denotes $M_8$, since it is the only one of relevance to the present article:
\begin{equation}
\begin{array}{llll}
M_8 & [z_1,z_2]=z_2, [z_3,z_4]=z_4.
\end{array}
\end{equation}

\section{The solvable subalgebras of $\sll$}\label{solvableb}

We proceed by dimension to classify the solvable subalgebras of $\sll$, up to inner automorphism.  The results 
are summarized in Table \ref{enuuu}.

\begin{theorem}
The classification of $1$-dimensional subalgebras in $\sll$, up to inner automorphism, is given below  
\begin{equation}\label{try2}
\begin{array}{llllllllll}
J^1&=& \langle x_1\rangle, &
J^2&=& \langle x_2\rangle,\\ 
J^3
&=& 
\langle h_1\rangle, 
&
J^{4}&=& \langle h_2\rangle, \\
J^{5}&=& \langle x_1+x_2\rangle, \\
J^{6}&=& \langle x_1+h_2\rangle, &
J^{7}&=& \langle h_1+x_2\rangle, \\
 J^{8, a}&=& \langle h_1+a h_2\rangle,  
 \quad for \,\, a \in \mathbb{C^{*}},
\end{array}
\end{equation}
where $J^{8, a} \sim J^{8, b}$ if and only if $a= \pm b$.
\end{theorem}
\begin{proof}
A nonzero element of $\langle x_1, h_1, y_1\rangle$ (resp. $\langle x_2, h_2, y_2\rangle$) is conjugate  
by an element of 
$GL(2, \mathbb{C})$, and hence $SL(2,\mathbb{C})$, to precisely one of $x_1$
or 
$ 
\pm
\alpha h_1$ (resp.  $x_2$
or $ 
\pm
\alpha h_2$), for 
$\alpha \in \mathbb{C}^{*}$.
Hence, a $1$-dimensional subalgebra of $\sll$ is conjugate to one of the following:  
\begin{equation}\label{try3}
\begin{array}{llllll}
\langle x_1\rangle&=&J^1,\\
\langle x_2\rangle&=&J^2,\\
\langle \alpha h_1\rangle =  \langle h_1\rangle
&=&
J^3,
\\
\langle \alpha h_2\rangle =  \langle h_2\rangle&=&J^4,\\
 \langle x_1+x_2\rangle&=&J^5,\\   
  \langle x_1+\alpha h_2\rangle \sim \langle x_1+h_2\rangle&=&J^6,\\
    \langle \alpha h_1+ x_2\rangle \sim  \langle h_1+x_2\rangle&=&J^7,\\
     \langle \alpha h_1+\beta h_2 \rangle =  \langle h_1+a h_2\rangle&=&J^{8,a},
\end{array}
\end{equation}
for $\alpha, \beta
\in \mathbb{C}^*$, 
$a=\frac{\beta}{\alpha}$.  Hence, Eq. \eqref{try2} is a complete list of $1$-dimensional subalgebras of $\sll$. 

Considering that the automorphism that ``switches" the $\mathfrak{sl}(2, \mathbb{C})$ components of $\sll$ is outer, the subalgebras  $J^1$, $J^2$, $J^3$, $J^4$, $J^5$, $J^6$, 
 $J^{7}$, $J^{8,a}$  
in Eq. \eqref{try3}  are pairwise inequivalent  for any fixed $a$.    It just remains to show 
$J^{8, a} \sim J^{8, b}$ 
if and only if $a=\pm b$.  

Suppose 
$J^{8,a}\sim J^{8, b}$;
then $h_1+a h_2$ is conjugate 
by an element of
$SL(2, \mathbb{C})\times SL(2, \mathbb{C})$ to 
$\lambda(h_1+b h_2)$ for some $\lambda\in \mathbb{C}$.  This implies that $h_1$ is conjugate to $\lambda h_1$, so that $\lambda=
\pm 1
$.
Thus $h_1+a h_2$ is conjugate to $h_1+b h_2$ 
by an element of
$SL(2, \mathbb{C})\times SL(2, \mathbb{C})$, which implies $a h_2$ is conjugate to $b h_2$ 
by an element of
$SL(2, \mathbb{C})$.  Thus $a= \pm b$.  Conversely, it is clear that if $a=\pm b$, then 
$J^{8, a} \sim J^{8, b}$.
\end{proof}

\begin{theorem}
There are precisely four subalgebras isomorphic to $K_1$  in $\sll$, up to inner automorphism.  They are:
\begin{equation}
\begin{array}{llllll}
K^1_1= \langle x_1, x_2 \rangle, &K^2_1= \langle x_1, h_2 \rangle,&
K^3_1= \langle h_1, x_2 \rangle,& 
K^{4}_1= \langle h_1,  h_2 \rangle.
\end{array}
\end{equation}
\end{theorem}
\begin{proof}
Let $\langle u, v\rangle$ be a $2$-dimensional, abelian subalgebra of $\sll$,  that is, a subalgebra isomorphic to $K_1$.  

Let $u =u_1+v_1$, and $v=v_1+v_2$, where $u_1, v_1 \in \langle x_1, h_1, y_1\rangle\cong \s$, and $u_2, v_2 \in \langle x_2, h_2, y_2\rangle\cong \s$.
Since $[u, v]=0$, then $[u_1, v_1]=[u_2, v_2]=0$.  Hence, $u_1$ and $v_1$ are linearly dependent, as are $u_2$ and $v_2$,
 since $\s$ does not contain an abelian subalgebra of dimension greater than $1$.  

 This implies that $\langle u, v \rangle$ must have a basis with one element in one $\s$ component of $\sll$, and the other element in the other $\s$ 
 component of $\sll$.  Without loss of generality, we may assume $u \in \langle x_1, h_1, y_1 \rangle$, and $v \in \langle x_2, h_2, y_2\rangle$.

Hence,  
as elements of $\s$,  
$u$ is conjugate by an element of $SL(2, \mathbb{C})$  
to (a nonzero scalar multiple of) exactly one of
$x_1$
or $h_1$,
and $v$ is conjugate 
to (a nonzero scalar multiple of) exactly one of $x_2$
 or $h_2$.  Hence, $\langle u, v\rangle$ is conjugate 
 by an element of 
 $SL(2,\mathbb{C})\times SL(2,\mathbb{C})$ to  one of $K^i_1$, $1\leq i \leq 4$. 
 Further, 
 these subalgebras are clearly pairwise inequivalent. 
\end{proof}

\begin{theorem}\label{ktwo}
Up to inner automorphism, 
there are five families of  subalgebras  
of $\sll$ that are 
isomorphic to $K_2$; two of them 
are infinite: 
\begin{equation}\label{ktt}
\begin{array}{llllllll}
K^{1}_2 &=& \langle x_1+x_2, h_1+h_2 
\rangle, &K^{2,a}_2&=& \langle x_1,  h_1 + a h_2  \rangle,\\
K^{3}_2&=& \langle x_1,  h_1 + x_2 \rangle, &K^{4,a}_2&=& \langle x_2,  h_2 + a h_1  \rangle,\\
K^{5}_2&=& \langle x_2,  h_2 + x_1  \rangle,
\end{array}
\end{equation}
for $a \in \mathbb{C}$, and such that 
$K^{i,a}_2 \sim K^{i,b}_2$ if and only if $a =  \pm b$ for $i=2, 4$.
\end{theorem}

\begin{proof}
 Let $\varphi: K_2 \hookrightarrow \sll$ be an embedding.
Up to inner automorphism, we may assume $\varphi(K_2) \subseteq \langle  x_1, x_2, h_1, h_2 \rangle$, since a Borel 
subalgebra is unique up to inner automorphism, and $\langle  x_1, x_2, h_1, h_2 \rangle$ is a Borel subalgebra of $\sll$.  Considering the commutation relations of 
$K_2$:
\begin{equation}
\begin{array}{llllll}
\varphi(z_1) &=& 
s
 x_1+
t
x_2,\\
\varphi(z_2) &=& c h_1+d h_2 +e x_1+fx_2,
\end{array}
\end{equation}
for $
s,t, c, d, e, f \in \mathbb{C}$, and 
$
s,t
$ not 
both
zero.    We now consider cases based on the values of $
s
$ and $
t
$.

\noindent \underline{Case 1}.  $
s
 \neq 0$, $
t
\neq 0$. Then, the commutation relations of $K_2$ imply that $c=d=-\frac{1}{2}$.  After conjugation by 
\begin{equation}
\begin{array}{l}
 \left(
\begin{array}{rrrr}
\alpha & - \alpha e & 0  &0 \\
 0 &  \alpha 
s
 & 0&  0 \\
 0& 0 &\beta  &-\beta  f\\
 0&0&0& \beta 
t
\end{array}
\right) \in SL(2,\mathbb{C}) \times SL(2,\mathbb{C}), 
\end{array} 
\end{equation}
where $\alpha$ is such that $\alpha^2 
s
=1$ and $\beta$ is such that $\beta^2 
t
=1$, 
we may assume
\begin{equation}
\begin{array}{llllll}
\varphi(z_1) &=& x_1+x_2,\\
\varphi(z_2) &=& -\frac{1}{2}( h_1+ h_2).
\end{array}
\end{equation}
Hence,  $\varphi(K_2)$ is equivalent to the subalgebra
\begin{equation}
\begin{array}{llllll}
K^1_2=\langle x_1+x_2,h_1+h_2\rangle.
\end{array}
\end{equation}

\noindent \underline{Case 2}. $
s
 \neq 0$ and $
t
=0$.  The commutation relations imply $c=-\frac{1}{2}$.  If $d=f=0$, then after conjugation by
\begin{equation}
\begin{array}{l}
 \left(
\begin{array}{rrrr}
 \alpha & -\alpha e & 0  &0 \\
 0 & \alpha 
s
 & 0&  0 \\
 0& 0 &1  &0\\
 0&0&0&1
\end{array}
\right) \in SL(2,\mathbb{C}) \times SL(2,\mathbb{C}), 
\end{array} 
\end{equation}
where  $\alpha$ is such that $\alpha^2 
s
=1$, we may assume
\begin{equation}
\begin{array}{llllll}
\varphi(z_1) &=& x_1,\\
\varphi(z_2) &=& -\frac{1}{2} h_1.
\end{array}
\end{equation}
Hence,  $\varphi(K_2)$ is equivalent to the subalgebra 
\begin{equation}
\begin{array}{llllll}
K^{2,0}_{2}=\langle x_1,h_1\rangle.
\end{array}
\end{equation}

If $d=0$, $f\neq 0$, then, after conjugation by
\begin{equation}
\begin{array}{l}
 \left(
\begin{array}{rrrr}
 \alpha & -\alpha e-\frac{1}{2}e \alpha 
s
 & 0  &0 \\
 0 & \alpha 
s
& 0&  0 \\
 0& 0 &\beta  &0\\
 0&0&0&-2\beta f
\end{array}
\right) \in SL(2,\mathbb{C}) \times SL(2,\mathbb{C}), 
\end{array} 
\end{equation}
where $\alpha$ is such that $\alpha^2
s
=1$ and $\beta$ is such that $-2\beta^2f=1$, we may assume
\begin{equation}
\begin{array}{llllll}
\varphi(z_1) &=& x_1,\\
\varphi(z_2) &=& -\frac{1}{2} (h_1
+x_2).
\end{array}
\end{equation}
Hence,  $\varphi(K_2)$ is equivalent to the subalgebra
\begin{equation}
\begin{array}{llllll}
K^{3}_{2}=\langle x_1,h_1+x_2\rangle.
\end{array}
\end{equation}

If $d\neq 0$, $f\neq 0$, then after conjugation by 
\begin{equation}
\begin{array}{l}
 \left(
\begin{array}{rrrr}
 \alpha & -\alpha e+e \alpha 
s
 & 0  &0 \\
 0 & \alpha 
s
 & 0&  0 \\
 0& 0 &1  &\frac{f}{2d}\\
 0&0&0&1
\end{array}
\right) \in SL(2,\mathbb{C}) \times SL(2,\mathbb{C}), 
\end{array} 
\end{equation}
where  $\alpha$ is such that $\alpha^2 
s
=1$, we may assume
\begin{equation}
\begin{array}{llllll}
\varphi(z_1) &=& x_1,\\
\varphi(z_2) &=& -\frac{1}{2} h_1+dh_2. 
\end{array}
\end{equation}
Hence (letting $a= -2d$), $\varphi(K_2)$ is equivalent to the subalgebra 
\begin{equation}
\begin{array}{llllll}
K^{2,a}_{2}=\langle x_1,h_1+ a h_2\rangle.
\end{array}
\end{equation}

\noindent \underline{Case 3}. $
s
=0$, $
t
\neq 0$. This case follows as in Case 1 to yield $K^{4,a}_{2}$, or $K^{5}_{2}$.  

The above cases imply that  a subalgebra of $\sll$ isomorphic to $K_2$ is equivalent to (at least) one of the subalgebras in Eq. \eqref{ktt}.  
We now show that the subalgebras
in Eq. \eqref{ktt} 
are
pairwise inequivalent (except possibly $K^{2,a}_2$ for different $a$, and $K^{4,a}_2$ for different $a$).   

Consider $K^1_2$ and $K^{2,a }_2$:  If $A\in SL(2,\mathbb{C})\times SL(2,\mathbb{C})$, such that $A \langle x_1+ x_2, h_1+h_2\rangle A^{-1}$ 
$=$ $\langle x_1, h_1+a h_2\rangle$, then $A(x_1+x_2)A^{-1}\in \langle x_1, h_1+ a h_2\rangle$, which implies that $x_2$ is conjugate to
 a multiple of $a h_2$ in $SL(2,\mathbb{C})$, which is not possible for any $a $.  Hence, $K^1_2 \nsim K^{2,a }_2$.  Using similar
 reasoning we have $K^1_2 \nsim K^{3}_2$; $K^1_2 \nsim K^{4,a  }_2$; $K^1_2 \nsim K^{5}_2$; $K^3_2 \nsim K^{2,a  }_2$; $K^3_2 \nsim K^{4,a }_2$;
 $K^3_2 \nsim K^{5}_2$;  $K^{2,a }_2 \nsim K^{4,b }_2$ for any $a,b $; $K^{2,a }_2 \nsim K^{5}_2$; and $K^{4,a }_2 \nsim K^{5}_2$.

We now consider the family of subalgebras $K^{2,a}_2$:    
Suppose $K^{2,a}_2 \sim K^{2,b}_2$.  Let $A\in SL(2,\mathbb{C})\times SL(2,\mathbb{C})$, such that $A \langle x_1, h_1+ a h_2\rangle A^{-1}$ $=$ $\langle x_1, h_1+
b h_2\rangle$.  Hence, $A(h_1+ 
a h_2)A^{-1}=\alpha x_1+\beta (h_1+ b h_2)$, for some $\alpha, \beta \in \mathbb{C}$, which implies that $h_1$ is conjugate in $SL(2,\mathbb{C})$ to $\beta h_1+\alpha x_1$.  Consideration of eigenvalues implies $\beta=\pm 1$.  

Thus, we have $A(h_1+ a h_2)A^{-1}=\alpha x_1 \pm (h_1+ b h_2)$, so that $
a h_2$ is conjugate to $\pm b h_2$.  Consideration of eigenvalues 
implies $
a=\pm b$. Further, the subalgebras 
$K^{2,a}_2$ and  $K^{2,-a}_2$ are conjugate via
\begin{equation}
\begin{array}{l}
 \left(
\begin{array}{rrrr}
1 &0 & 0  &0 \\
 0 & 1 & 0&  0 \\
 0& 0 &0  &1\\
 0&0&-1&0
\end{array}
\right) \in SL(2,\mathbb{C}) \times SL(2,\mathbb{C}).
\end{array} 
\end{equation}
Hence, we have established that  $K^{2,a}_2 \sim K^{2,b}_2$ if and only if $a = \pm b$.  Similarly, $K^{4,a}_2 \sim K^{4,b}_2$ if and only if $a =\pm b $.
\end{proof}

\begin{theorem}
There are no subalgebras isomorphic to $L_1$ in  $\sll$.  
\end{theorem}
\begin{proof}
$L_1$ is a $3$-dimensional, abelian Lie algebra. The existence of a subalgebra isomorphic to $L_1$ in  $\sll$ would imply a $2$-dimensional
 abelian subalgebra in $\mathfrak{sl}(2,\mathbb{C})$, which is a contradiction.  
\end{proof}

\begin{theorem}
There is a unique $L_2$ subalgebra in $\sll$, up to inner automorphism, given by $L_2^1 = \langle x_1, x_2, h_1+h_2 \rangle$.
\end{theorem}
\begin{proof}
Consider an embedding $\varphi: L_2 \hookrightarrow \sll$.    Since $L_2$ is solvable, it is mapped into a Borel subalgebra of $\sll$.   Since there
is a unique Borel subalgebra in a semisimple Lie algebra, up to inner automorphism, we may assume
\begin{equation}
\varphi( L_2 )=\varphi(\langle z_1, z_2, z_3 \rangle)\subseteq \langle h_1, h_2, x_1, x_2\rangle.
\end{equation}
The commutation relations of $L_2$ imply
\begin{equation}
\varphi(\langle z_1, z_2 \rangle)\subseteq \langle  x_1, x_2\rangle.
\end{equation}
Hence
\begin{equation}
\begin{array}{llllllll}
\varphi(z_1)= 
s
x_1+ 
t
x_2,\\
\varphi(z_2)=cx_1+dx_2,\\
\varphi(z_3)=eh_1+fh_2+gx_1+hx_2,
\end{array}
\end{equation}
for $
s,t, c, d, e, f, g, h \in \mathbb{C}$.
The commutation relations imply $e=f=\frac{1}{2}$.  After conjugation  by
\begin{equation}
\begin{array}{l}
 \left(
\begin{array}{rrrr}
1 &g & 0  &0 \\
 0 & 1 & 0&  0 \\
 0& 0 &1  &h\\
 0&0&0&1
\end{array}
\right) \in SL(2,\mathbb{C}) \times SL(2,\mathbb{C}).
\end{array} 
\end{equation}
we may assume
\begin{equation}
\begin{array}{llllllll}
\varphi(z_1)= 
s
x_1+ 
t
x_2,\\
\varphi(z_2)=cx_1+dx_2,\\
\varphi(z_3)=\frac{1}{2}h_1+\frac{1}{2}h_2.
\end{array}
\end{equation}
Thus,  $\varphi(L_2)$ must be equivalent to $\langle x_1, x_2, h_1+h_2 \rangle=L^1_2$.
\end{proof}

\begin{theorem}
For each $a\in \mathbb{C}-\{-\frac{1}{4}, 0\}$, there are two subalgebras isomorphic to $L_{3,a}$ in $\sll$, up to inner automorphism:
\begin{equation}
\begin{array}{llllll}
L_{3, a}^{1} &=& \langle x_1, x_2, (1 + \sqrt{1+4a} ) h_1+ (1 - \sqrt{1+4a} ) h_2 \rangle, \\
L_{3, a}^{2} &=& \langle x_1, x_2, (1 - \sqrt{1+4a} ) h_1+(1 + \sqrt{1+4a} ) h_2 \rangle.
\end{array}
\end{equation} 
To be definite,  we assume the square roots in the above formulas are always in the upper half-plane 
or on the positive real axis.
There is no  subalgebra isomorphic to $L_{3,-\frac{1}{4}}$ in $\sll$.  There are four  subalgebras in $\sll$ isomorphic to $L_{3,0}$, up to inner automorphism:
 \begin{equation}
\begin{array}{llllllllllll}
L_{3, 0}^{1} &=& \langle   h_1, x_2, h_2\rangle,&
L_{3, 0}^{2} &=& \langle x_1, x_2, h_2 \rangle,\\
L_{3, 0}^{3} &=& \langle  x_1, h_1, h_2\rangle,&
L_{3, 0}^{4} &=& \langle  x_1, h_1, x_2\rangle.
\end{array}
\end{equation}
\end{theorem}
\begin{proof}
First assume $a\neq 0$ in $L_{3,a}$.  Consider an embedding $\varphi: L_{3,a} \hookrightarrow \sll$.    Since $L_{3,a}$ is solvable, it is mapped into a Borel subalgebra of $\sll$.   Since there is a unique Borel subalgebra in a semisimple Lie algebra, up to inner automorphism, we may assume
\begin{equation}
\varphi(\langle z_1, z_2, z_3 \rangle)\subseteq \langle h_1, h_2, x_1, x_2\rangle.
\end{equation}
The commutation relations of $L_{3,a}$ imply 
\begin{equation}
\begin{array}{llllll}
\varphi(z_1) &=& bx_1+cx_2,\\
\varphi(z_2) &=& (2bd)x_1+(2ce)x_2,\\
\varphi(z_3) &=& dh_1+eh_2+fx_1+gx_2,\\
\end{array}
\end{equation}
where we must have  $b, c \neq 0$ since $\varphi(z_1)$, and $\varphi(z_2)$  are linearly independent.  After conjugation by the 
 $SL(2,\mathbb{C}) \times SL(2,\mathbb{C})$ elements
\begin{equation}
\begin{array}{l}
 \left(
\begin{array}{rrrr}
\frac{2\alpha d}{f} &\alpha & 0  &0 \\
 0 & \frac{2b \alpha d}{f} & 0&  0 \\
 0& 0 &\frac{2\beta e}{g}  &\beta\\
 0&0&0&\frac{2\beta ec}{g}
\end{array}
\right), ~\text{if}~f, g\neq 0,
\end{array} 
\end{equation}
where  $\alpha$ is such that $4\alpha^2d^2b=f^2$, and $\beta$ is such that $4\beta^2e^2c=g^2$,
\begin{equation}
\begin{array}{l}
 \left(
\begin{array}{rrrr}
\alpha &0 & 0  &0 \\
 0 & \alpha b & 0&  0 \\
 0& 0 &\beta  &0\\
 0&0&0&\beta c
\end{array}
\right), ~\text{if}~ f, g=0,
\end{array} 
\end{equation}
where   $\alpha$ is such that $\alpha^2b=1$, and  $\beta$ is such that $\beta^2c=1$,
\begin{equation}
\begin{array}{l}
 \left(
\begin{array}{rrrr}
\alpha &0 & 0  &0 \\
 0 & \alpha b & 0&  0 \\
 0& 0 &\frac{2\beta e}{g}  &\beta\\
 0&0&0&\frac{2\beta ec}{g}
\end{array}
\right), ~\text{if}~ f=0, g\neq 0,
\end{array} 
\end{equation}
where  $\alpha$ is such that $\alpha^2b=1$, and   $\beta$ is such that $4\beta^2e^2c=g^2$,
\begin{equation}
\begin{array}{l}
 \left(
\begin{array}{rrrr}
\frac{2\alpha d}{f} & \alpha & 0  &0 \\
 0 & \frac{2b\alpha d}{f} & 0&  0 \\
 0& 0 &\beta  &0\\
 0&0&0&\beta c
\end{array}
\right), ~\text{if}~f\neq 0, g= 0,
\end{array} 
\end{equation}
where  $\alpha$ is such that $4\alpha^2d^2b=f^2$, and  $\beta$ is such that $\beta^2c=1$, we may assume
\begin{equation}
\begin{array}{llllll}
\varphi(z_1) &=& x_1+x_2,\\
\varphi(z_2) &=& (2d)x_1+(2e)x_2,\\
\varphi(z_3) &=& d h_1+eh_2.
\end{array}
\end{equation}
We must have $d\neq e$ for otherwise $\varphi(z_1)$ and $\varphi(z_2)$ would be linearly dependent.  The 
commutation relation $[z_3, z_2]=az_1+z_2$ implies
\begin{equation}
\begin{array}{llllll}
a=2d(2d-1)=2e(2e-1),
\end{array}
\end{equation}
so that 
\begin{equation}
\begin{array}{llllll}
d, e =  \frac{1}{4}(1 \pm \sqrt{1+4a} ). 
\end{array}
\end{equation}
Since $d\neq e$, then $a\neq -\frac{1}{4}$, hence $\sll$ does not have a subalgebra isomorphic 
to $L_{3,-\frac{1}{4}}$.  Further, for $a\neq -\frac{1}{4}, 0$, we thus must have that  $\varphi(L_{3,a})$ is equivalent to one of  the subalgebras
\begin{equation}
\begin{array}{llllll}
L_{3, a}^{1} &=& \langle x_1, x_2, (1 + \sqrt{1+4a} ) h_1+ (1 - \sqrt{1+4a} ) h_2 \rangle, \\
L_{3, a}^{2} &=& \langle x_1, x_2, (1 - \sqrt{1+4a} ) h_1+(1 + \sqrt{1+4a} ) h_2 \rangle.
\end{array}
\end{equation}

Let $A \in SL(2,\mathbb{C})\times SL(2,\mathbb{C})$, such that $A L^1_{3,a}A^{-1}=L^2_{3,a}$.   
This implies that   $Ax_1A^{-1}=\alpha x_1$ and $Ax_2A^{-1}=\beta x_2$, for some  $\alpha$, $\beta \in \mathbb{C}^*$ and 
\begin{equation}
\begin{array}{l}
A= \left(
\begin{array}{rrrr}
\alpha \gamma & \delta & 0  &0 \\
 0 & \gamma & 0&  0 \\
 0& 0 &\beta \epsilon  &\zeta\\
 0&0&0&\epsilon
\end{array}
\right),
\end{array} 
\end{equation}
for  $\gamma, \delta, \epsilon, \zeta \in \mathbb{C}$ such that $\alpha \gamma^2=1$, and $\beta \epsilon^2=1$. 
Hence
\begin{equation}\label{runnn}
\begin{array}{l}
A((1 + \sqrt{1+4a} ) h_1+(1 - \sqrt{1+4a} ) h_2)A^{-1}\\
 = (1 + \sqrt{1+4a} ) h_1 +(1 -\sqrt{1+4a} ) h_2  -\frac{2d\delta}{\gamma} x_1 -\frac{2e\zeta}{\epsilon}  x_2. 
\end{array} 
\end{equation}
We thus must have that $(1 + \sqrt{1+4a} ) h_1 +(1 - \sqrt{1+4a} ) h_2$ is a scalar multiple of 
$(1 - \sqrt{1+4a} )h_1 +(1 + \sqrt{1+4a} ) h_2$, since the element of Eq. \eqref{runnn} must  be in $L^2_{3,a}$.  Hence
\begin{equation}\label{runn}
\begin{array}{l}
(1 + \sqrt{1+4a} ) = \lambda (1 - \sqrt{1+4a} ), ~\text{and}\\
(1 - \sqrt{1+4a} ) = \lambda (1 + \sqrt{1+4a} ),
\end{array} 
\end{equation}
for some $\lambda \in \mathbb{C}$.  Hence $\lambda=\pm1$.  If $\lambda=1$, then $a=-\frac{1}{4}$, a contradiction.  A contradiction
also occurs for $\lambda=-1$.  Hence $L^1_{3,a} \nsim L^2_{3,a}$.

We now assume $a=0$.  Consider an embedding $\varphi: L_{3,0} \hookrightarrow \sll$;
then after applying an appropriate inner automorphism of $\sll$, we my assume:
\begin{equation}
\varphi(\langle z_1, z_2, z_3 \rangle)\subseteq \langle h_1, h_2, x_1, x_2\rangle.
\end{equation}
The commutation relations  of $L_{3,0}$ imply
\begin{equation}
\begin{array}{llllll}
\varphi(z_1) &=& bx_1+cx_2+dh_1+eh_2,\\
\varphi(z_2) &=& fx_1+gx_2,\\
\varphi(z_3) &=& hx_1+ix_2+jh_1+kh_2,\\
\end{array}
\end{equation}
for $b, c, d, e, f, g, h, i, j, k \in \mathbb{C}$.    We now consider cases based on the value of $f$ and $g$.  Note that, of course, 
$f$ and $g$ cannot both
be $0$.

\noindent \underline{Case 1}. $f, g \neq 0$.  Then, the commutation relation $[z_3, z_2]=z_2$ implies that $f(2j-1)=g(2k-1)=0$.  Hence,
$j=k=\frac{1}{2}$.  Since $[z_1,z_2]=0$, $d=e=0$.  The commutation relation $[z_3, z_1]=z_2$ then implies $b=f$ and $c=g$, so that 
$\varphi(z_1)=b x_1+c x_2 =\varphi(z_2)$, a contradiction.  Hence,  
$f$ and $g$ cannot both be nonzero.

\noindent \underline{Case 2}. $f=0, g \neq 0$. The commutation relations imply $jb=dh$, $k=\frac{1}{2}$, $e=0$, and $c=g$:
\begin{equation}
\begin{array}{llllll}
\varphi(z_1) &=& bx_1+gx_2+dh_1,\\
\varphi(z_2) &=& gx_2,\\
\varphi(z_3) &=& hx_1+ix_2+jh_1+\frac{1}{2}h_2.\\
\end{array}
\end{equation}
Note that 
 $b$ and $d$ cannot both
 be $0$, otherwise $\varphi(z_1)=\varphi(z_2)$.   Hence, the following subcases based on the possible 
values of $b$ and $d$ are exhaustive.  

\noindent \underline{Case 2.1}.  $b=0$, $d\neq 0$.  Since $jb=dh$, we must have $h=0$:
\begin{equation}
\begin{array}{llllll}
\varphi(z_1) &=& gx_2+dh_1,\\
\varphi(z_2) &=& gx_2,\\
\varphi(z_3) &=& ix_2+jh_1+\frac{1}{2}h_2.\\
\end{array}
\end{equation}
In this case $\varphi(L_{3,0})=\langle h_1, x_2, h_2 \rangle=L^1_{3,0}$.

\noindent \underline{Case 2.2}. $b\neq 0$, $d= 0$.  Since $jb=dh$, we must have $j=0$:
\begin{equation}
\begin{array}{llllll}
\varphi(z_1) &=& bx_1+gx_2,\\
\varphi(z_2) &=& gx_2,\\
\varphi(z_3) &=& hx_1+ix_2+\frac{1}{2}h_2.\\
\end{array}
\end{equation}
In this case $\varphi(L_{3,0})=\langle x_1, x_2, h_2 \rangle=L^2_{3,0}$.

\noindent \underline{Case 2.3.1}. $b, d\neq 0$, $j=h=0$:
\begin{equation}
\begin{array}{llllll}
\varphi(z_1) &=& bx_1+gx_2+dh_1,\\
\varphi(z_2) &=& gx_2,\\
\varphi(z_3) &=& ix_2+\frac{1}{2}h_2.\\
\end{array}
\end{equation}
Further,
\begin{equation}
\begin{array}{llllll}
A\varphi(z_1)A^{-1} &=& gx_2+dh_1,\\
A\varphi(z_2) A^{-1}&=& gx_2,\\
A\varphi(z_3)A^{-1} &=& ix_2+\frac{1}{2}h_2.\\
\end{array}
\end{equation}
where 
\begin{equation}\label{kjhr}
\begin{array}{l}
A= \left(
\begin{array}{rrrr}
\frac{2d}{b} &1 & 0  &0 \\
 0 & \frac{b}{2d} & 0&  0 \\
 0& 0 &1  &0\\
 0&0&0&1
\end{array}
\right) \in SL(2,\mathbb{C}) \times SL(2,\mathbb{C}).
\end{array} 
\end{equation}
Hence, $\varphi(L_{3,0})\sim \langle h_1, x_2, h_2 \rangle=L^1_{3,0}$.

\noindent \underline{Case 2.3.2}. $b, d\neq 0$,  $j, h\neq 0$. 
\begin{equation}
\begin{array}{llllll}
\varphi(z_1) &=& bx_1+gx_2+dh_1,\\
\varphi(z_2) &=& gx_2,\\
\varphi(z_3) &=& hx_1+ix_2+jh_1+\frac{1}{2}h_2.\\
\end{array}
\end{equation}
Then, for $A$ of Eq. \eqref{kjhr}, 
\begin{equation}
\begin{array}{llllll}
A\varphi(z_1)A^{-1} &=& gx_2+dh_1,\\
A\varphi(z_2)A^{-1} &=& gx_2,\\
A\varphi(z_3) A^{-1}&=& ix_2+jh_1+\frac{1}{2}h_2.\\
\end{array}
\end{equation}
Hence we have $\varphi(L_{3,0})\sim \langle h_1, x_2, h_2 \rangle=L^1_{3,0}$.   Note that if $b, d\neq 0$, either both $j$ and $h$ are nonzero or both are 
$0$, since $jb=dh$.

\noindent \underline{Case 3}. $f\neq 0, g = 0$. This case follows as in Case 2 to yield 
the subalgebras $L^3_{3,0}=\langle x_1, h_1, h_2\rangle$, and $L^4_{3,0}=\langle  x_1, h_1, x_2\rangle$.

The above cases imply that any subalgebra of $\sll$ isomorphic to $L_{3,0}$ is equivalent 
to $L^1_{3,0}$, $L^2_{3,0}$, $L^3_{3,0}$, or $L^4_{3,0}$.   It remains to show that these subalgebras are pairwise inequivalent.

$L^1_{3,0}$ is not equivalent to $L^2_{3,0}$, since if $L^1_{3,0} \sim L^2_{3,0}$, then, for some 
$A \in SL(2,\mathbb{C})\times SL(2,\mathbb{C})$, $AL^{2}_{3,0}A^{-1}=L^{1}_{3,0}$.  This 
implies $Ax_1A^{-1}$ is a scalar multiple of $h_1$, which of course is not possible.  The remaining inequivalency cases follow in a similar fashion.
\end{proof}

\begin{theorem}
There is a unique $L_4$ subalgebra in $\sll$, up to inner automorphism:  $L_4^1 = \langle x_1, x_2, h_1-h_2 \rangle$.
\end{theorem}
\begin{proof}
$L_4$ is isomorphic to the  complexification of the Euclidean algebra $\e_{\mathbb{C}} \cong \mathfrak{so}(2,\mathbb{C}) \inplus \mathbb{C}^2$.
The classification of embeddings of $\e_{\mathbb{C}} $ into $\sll$ was given in \cite{dr1}.  It was shown that are two embeddings $\varphi_1$, $\varphi_2$, up to inner automorphism, such that 
$\varphi_1(\e_{\mathbb{C}} )=\varphi_2(\e_{\mathbb{C}} )=\langle x_1, x_2, h_1-h_2  \rangle$.  The result follows.
\end{proof}

\begin{theorem}
There are no subalgebras isomorphic to $L_5$  in $\sll$.  
\end{theorem}
\begin{proof}
$L_5$ is nilpotent.  If there were a subalgebra isomorphic to $L_5$ in $\sll$, then $\s$ would have a nilpotent subalgebra of dimension at least $2$, which is a contradiction.  Hence,
there is no subalgebra isomorphic to  $L_5$ in $\sll$.
\end{proof}

\begin{theorem}
There is a unique $4$-dimensional, solvable subalgebra in $\sll$, up to inner automorphism: 
\begin{equation}
M^1_8 \cong \langle x_1, x_2, h_1, h_2 \rangle.
\end{equation}
\end{theorem}
\begin{proof}
 The $4$-dimensional subalgebra $\langle x_1, x_2, h_1, h_2 \rangle$ of $\sll$ is a Borel subalgebra.  Hence it is the unique solvable subalgebra of $\sll$ of dimension $4$, up to inner automorphism.  Further, $M_8$ $\cong$ $\langle x_1, x_2, h_1, h_2 \rangle$  via the isomorphism:   $z_2 \mapsto x_1$;  $z_1 \mapsto \frac{1}{2}h_1$ ; $z_4 \mapsto x_2$; $z_3 \mapsto \frac{1}{2}h_2$.
\end{proof}

Since a Borel subalgebra of $\sll$, which is of dimension $4$ (and unique up to inner automorphism), is  a maximal solvable subalgebra, we have the following theorem, which completes the 
classification of solvable subalgebras of $\sll$.
\begin{theorem}
There are no solvable subalgebras in $\sll$ of dimension greater than $4$.
\end{theorem}

\begin{table} [!h]\renewcommand{\arraystretch}{1.6} \caption{Classification of solvable subalgebras of $\sll$, up to inner automorphism.} \label{enuuu}\begin{center}
\begin{tabular}{|c|c|c|clclcl} 
\hline
Dimension &  Solvable subalgebras \\
\hline \hline
1 & 
$J^1$, $J^2$, $J^3$, $J^4$, $J^5$, $J^6$,  
$J^7$,  $J^{8, a}$ , $a \in \mathbb{C}^{*}$
\\
&
($J^{8,a} \sim J^{8,a'}$ iff $a'  =  \pm a$)
\\
\hline 
2& $K_1^i$, $1\leq i \leq 4$;  $K^1_2$, $K^{2,a}_2$, $K^{3}_2$, $K^{4,a}_2$, $K^{5}_2$, $a \in \mathbb{C}$\\
&
 (for $i=2, 4$:  $K^{i,a}_2 \sim K^{i,a'}_2$ iff $a'  =  \pm a$)
 \\
\hline 
3 &    $L_2^1$,  $L_{3, a}^{1}$,  $L_{3, a}^{2}$, $L_{3, 0}^{i}$, $L_4^1$,  \\
& $a\in \mathbb{C}-\{-\frac{1}{4}, 0 \}$, $1\leq i \leq 4$\\
\hline
4 &   $M_8^1$\\
\hline
\end{tabular}\end{center}
\end{table}

\section{The semisimple subalgebras of $\sll$}\label{semi}

\begin{theorem}\label{semmii}
There are three 
proper
semisimple subalgebras of $\sll$, up to inner automorphism, and each is isomorphic to $\s$.  They are:
 \begin{equation}
 \begin{array}{llllllllll}
A^1_1 = \langle h_1, x_1, y_1  \rangle, & A^3_1 = \langle h_1+h_2, x_1+x_2, y_1+ y_2 \rangle,\\
A^2_1 = \langle h_2, x_2, y_2 \rangle.     
 \end{array}
 \end{equation}
 \end{theorem}
\begin{proof}
The only possible proper semisimple subalgebra of $\sll$ is $\s$.
Let $\varphi:  \s \hookrightarrow \sll$ be an embedding and $\{ h, x, y\}$ a Chevalley basis of $\s$ with $[h,x]=2x$, $[h,y]=-2y$, and $[x,y]=h$.  
Then $\varphi(\langle h, x \rangle)\cong K_2$, hence, by Theorem \ref{ktwo}, it must be equivalent to $K^1_2$, $K^{2,a}_2$, $K^3_2$, $K^{4,a}_2$, or $K^5_2$, for some $a \in \mathbb{C}$.  We proceed in cases.

\noindent \underline{Case 1}.  $\varphi(\langle h, x \rangle)\sim K^1_2$.  The commutation relation $[h, x]=2x$ implies, after appropriate inner automorphism,
 \begin{equation}
 \begin{array}{llll}
 \varphi(x)&=& a(x_1+x_2), a\neq 0,\\
  \varphi(h)&=& b(x_1+x_2)+(h_1+h_2).
 \end{array}
 \end{equation}
Then, after conjugation by 
\begin{equation}
\begin{array}{l}
 \left(
\begin{array}{rrrr}
\alpha &\frac{\alpha b}{2} & 0  &0 \\
 0 & \alpha a & 0&  0 \\
 0& 0 &\alpha  &\frac{\alpha b}{2}\\
 0&0&0&\alpha a
\end{array}
\right) \in SL(2,\mathbb{C}) \times SL(2,\mathbb{C}),
\end{array} 
\end{equation}
where  $\alpha$ is such that $\alpha^2 a=1$, we may assume
 \begin{equation}
 \begin{array}{llll}
 \varphi(x)&=& x_1+x_2,\\
  \varphi(h)&=& h_1+h_2.
 \end{array}
 \end{equation}
The commutation relations $[h, y]=-2y$ and $[x,y]=h$ imply
 \begin{equation}
 \begin{array}{llll}
 \varphi(y)&=& y_1+y_2.
 \end{array}
 \end{equation}
Hence, in this case $\varphi(\s)$ is equivalent to $A^3_1=\langle h_1+h_2, x_1+x_2, y_1+y_2\rangle$.

\noindent \underline{Case 2}.  $\varphi(\langle h, x \rangle)\sim K^{2,a}_2$. 
The commutation relation $[h, x]=2x$ implies, after appropriate inner automorphism,
 \begin{equation}
 \begin{array}{llll}
 \varphi(x)&=& bx_1, b\neq 0,\\
  \varphi(h)&=& c x_1+h_1+ah_2.
 \end{array}
 \end{equation}
Then, after conjugation by 
\begin{equation}
\begin{array}{l}
 \left(
\begin{array}{rrrr}
\alpha &\frac{\alpha c}{2} & 0  &0 \\
 0 & \alpha b & 0&  0 \\
 0& 0 &1  &0\\
 0&0&0&1
\end{array}
\right) \in SL(2,\mathbb{C}) \times SL(2,\mathbb{C}),
\end{array} 
\end{equation}
where  $\alpha$ is such that $\alpha^2b=1$, we may assume
 \begin{equation}
 \begin{array}{llll}
 \varphi(x)&=& x_1,\\
  \varphi(h)&=& h_1+ah_2.
 \end{array}
 \end{equation}
The commutation relations $[h, y]=-2y$ and $[x,y]=h$ imply  that $a=0$ and 
 \begin{equation}
 \begin{array}{llll}
 \varphi(y)&=& y_1.
 \end{array}
 \end{equation}
Hence, in this case $\varphi(\s)$ is equivalent to $A^1_1=\langle h_1, x_1, y_1\rangle$.

\noindent \underline{Case 3}.  $\varphi(\langle h, x \rangle)\sim K^{4,a}_2$.  This case follows as in Case 2 to yield that 
$\varphi(\s)$ is equivalent to $A^2_1=\langle h_2, x_2, y_2\rangle$.

\noindent \underline{Case 4}.  $\varphi(\langle h, x \rangle)\sim K^{3}_2$. 
The commutation relation $[h, x]=2x$ implies, after appropriate inner automorphism,
 \begin{equation}
 \begin{array}{llll}
 \varphi(x)&=& ax_1, a\neq 0,\\
  \varphi(h)&=& b x_1+h_1+x_2.
 \end{array}
 \end{equation}
Then, after conjugation by 
\begin{equation}
\begin{array}{l}
 \left(
\begin{array}{rrrr}
\alpha &\frac{\alpha b}{2} & 0  &0 \\
 0 & \alpha a & 0&  0 \\
 0& 0 &1  &0\\
 0&0&0&1
\end{array}
\right) \in SL(2,\mathbb{C}) \times SL(2,\mathbb{C}),
\end{array} 
\end{equation}
where  $\alpha$ is such that $\alpha^2a=1$, we may assume
 \begin{equation}
 \begin{array}{llll}
 \varphi(x)&=& x_1,\\
  \varphi(h)&=& h_1+x_2.
 \end{array}
 \end{equation}
The  commutation relation $[h, y]=-2y$ implies that $\varphi(y)$ is a nonzero scalar multiple of $y_1$.  However, this implies  $\varphi([x,y]) \neq [\varphi(x), \varphi(y)]$, a contradiction to $\varphi$ being a Lie algebra embedding.
Hence, this case cannot yield a subalgebra isomorphic to $\s$.

\noindent \underline{Case 5}.  $\varphi(\langle h, x \rangle)\sim K^{5}_2$. It follows as in Case 4 that 
no subalgebra isomorphic to $\s$ comes from this case.

The above cases imply that the every subalgebra of $\sll$ isomorphic to $\s$--that is, of type $A_1$--is equivalent to 
$A^1_1$, $A^2_1$, or $A^3_1$.  Further, these subalgebras are clearly pairwise inequivalent.
\end{proof}

\section{The Levi decomposable subalgebras of $\sll$}\label{les}

Below we decompose $\sll$ with respect to the adjoint action of the subalgebras $A_1^1$, $A_1^2$, and $A_1^3$, respectively.  

\begin{equation}\label{acc3}
\begin{array}{ccccccccccccccc}
\sll &\cong_{A^1_1}& \langle x_1, h_1, y_1 \rangle&\oplus& \langle x_2 \rangle&\oplus&
\langle y_2\rangle&\oplus&  \langle h_2 \rangle \\
&  \cong_{A^1_1} & V_{}(2) &\oplus&V_{}(0)&\oplus&V_{}(0)&\oplus& V_{}(0),
\end{array}
\end{equation}
\begin{equation}\label{acc4}
\begin{array}{cccccccccccccccc}
\sll &\cong_{A^2_1}&  \langle x_2, h_2, y_2 \rangle&\oplus& \langle x_1 \rangle&\oplus&
\langle y_1\rangle&\oplus&  \langle h_1 \rangle \\
&  \cong_{A^2_1} & V_{}(2) &\oplus&V_{}(0)&\oplus&V_{}(0)&\oplus& V_{}(0),
\end{array}
\end{equation}
\begin{equation}\label{acc5}
\begin{array}{cccccccccccccccc}
\sll &\cong_{A^3_1}&  \langle x_1+x_2, h_1+h_2, y_1+y_2 \rangle&\oplus& \langle x_2, h_2, y_2 \rangle \\
&  \cong_{A^3_1} & V_{}(2) &\oplus&V_{}(2).
\end{array}
\end{equation}
The components--or combinations thereof--in the decompositions above, beyond $A_1^1$, $A_2^1$, and $A_3^1$, respectively, give us the possible 
extensions of $A_1^1$, $A_2^1$, and $A_3^1$, respectively.

\begin{lemma}\label{lem1}
Let  $\psi: \la \inplus \mathfrak{s} \rightarrow \mathfrak{h} \inplus \mathfrak{r}$ be a Lie algebra isomorphism, where  $\mathfrak{g}$  and  $\mathfrak{h}$  are semisimple and  $\mathfrak{s}$  and  $\mathfrak{r}$  are solvable.
Then $\psi(\mathfrak{s}) = \mathfrak{r}$.
\end{lemma}
\begin{proof}
Let $\pi: \mathfrak{h} \inplus \mathfrak{r} \rightarrow  \mathfrak{h}$ be the projection map of  $\mathfrak{h} \inplus \mathfrak{r}$ onto  $\mathfrak{h}$.  Then,
$\pi( \psi(\mathfrak{s}))$ is a solvable ideal of $\mathfrak{h}$.   Since  $\mathfrak{h}$ is semisimple, $\pi( \psi(\mathfrak{s}))=0$.  Hence,  $\psi(\mathfrak{s})\subseteq \mathfrak{r}$.  The Levi factor in a Levi decomposition is unique, up to isomorphism (i.e., $\la \cong\mathfrak{h}$).  Hence, dimension considerations imply $\psi(\mathfrak{s}) = \mathfrak{r}$.
\end{proof}

\begin{theorem}\label{ld1}
There are four $4$-dimensional Levi decomposable subalgebras of $\sll$, up to inner automorphism.  They are:
\begin{equation}
\begin{array}{llllllllll}
(A_1 \oplus J)^1&=& \langle x_1, y_1, h_1 \rangle \oplus \langle h_2 \rangle, \\
(A_1 \oplus J)^2&=& \langle x_1, y_1, h_1 \rangle \oplus \langle x_2 \rangle, \\
(A_1 \oplus J)^3&=& \langle x_2, y_2, h_2 \rangle \oplus \langle h_1 \rangle, \\
(A_1 \oplus J)^4&=& \langle x_2, y_2, h_2 \rangle \oplus \langle x_1 \rangle.
\end{array}
\end{equation}
The classification is recorded in Table \ref{enggg}. 
\end{theorem}
\begin{proof}
The only possible $4$-dimensional Levi decomposable subalgebra of $\sll$ is isomorphic 
to $\s \oplus J$.  Note that the sum must be direct, which is the case whenever we extend a semisimple
Lie algebra by a one dimensional algebra to form a Levi decomposable algebra.

From Theorem \ref{semmii}, any subalgebra of $\sll$ isomorphic to $\s$ is equivalent to exactly one of
$A^1_1=\langle x_1, y_1, h_1 \rangle$,  $A^2_1=\langle x_2, y_2, h_2 \rangle$, or $A^3_1=\langle x_1+x_2, y_1+y_2, h_1+h_2 \rangle$.
We now must determine the ways that we may extend each of these simple subalgebras of $\sll$ by a 
$1$-dimensional subalgebra.

Eq. \eqref{acc5} describes the adjoint action of $A^3_1$ on $\sll$.  There are two components, one is $A^3_1$ itself, and both are 
isomorphic to $V(2)$ as  $\s$-representations.  Thus, the only nontrivial extension of 
$A^3_1$ is all of $\sll$.  Hence, the subalgebra $A^3_1$ cannot be extended by a $1$-dimensional subalgebra.

The possible extensions of the subalgebra $A^1_1$ by a $1$-dimensional subalgebra are given by Eq. \eqref{acc3}.  Namely
$A^1_1 \oplus \langle a x_2+bh_2+cy_2\rangle$, for all  $a, b, c \in \mathbb{C}$, not all $0$.  However, 
\begin{equation}
\begin{array}{lllllll}
A^1_1 \oplus \langle a x_2+bh_2+cy_2\rangle &\sim& A^1_1 \oplus \langle h_2\rangle&=&(A_1 \oplus J)^1, ~\text{or}\\
A^1_1 \oplus \langle a x_2+bh_2+cy_2\rangle &\sim& A^1_1 \oplus \langle x_2\rangle&=&(A_1 \oplus J)^2.
\end{array}
\end{equation}
The equivalencies are given by a matrix of the form
\begin{equation} \label{aamhj}
\begin{array}{l}
 \left(
\begin{array}{rrrr}
I &0  \\
 0 & M
\end{array}
\right) \in SL(2,\mathbb{C}) \times SL(2,\mathbb{C}),
\end{array} 
\end{equation}
where $M\in SL(2,\mathbb{C})$, and $I$ is the $2\times 2$ identity. The matrix $M$ is chosen so that it conjugates $a x_2+bh_2+cy_2$ into its 
Jordan normal form.

Similarly, an extension of $A^2_1$ by a one dimensional subalgebra of $\sll$ is equivalent to 
\begin{equation}
\begin{array}{lllllll}
A^2_1 \oplus \langle h_1\rangle&=&(A_1 \oplus J)^{ 3}, ~\text{or} &
A^2_1 \oplus \langle x_1\rangle&=&(A_1 \oplus J)^{ 4}.
\end{array}
\end{equation}

The subalgebras $(A_1 \oplus J)^1$, $(A_1 \oplus J)^2$, $(A_1 \oplus J)^3$, and $(A_1 \oplus J)^4$ are easily seen to be pairwise inequivalent.
\end{proof}

\begin{theorem}\label{ld2}
There are two $5$-dimensional Levi decomposable subalgebras of $\sll$, up to inner automorphism.   They are:
\begin{equation}
\begin{array}{llllllllll}
(A_1 \oplus K_2)^1&=& \langle x_1, y_1, h_1 \rangle \oplus \langle x_2,  h_2 \rangle, \\
(A_1 \oplus K_2)^2&=& \langle x_2, y_2, h_2 \rangle \oplus \langle x_1, h_1 \rangle.
\end{array}
\end{equation}
The classification is recorded in Table \ref{enggg}. 
\end{theorem}
\begin{proof}
From Theorem \ref{semmii}, any subalgebra of $\sll$ isomorphic to $\s$ is equivalent to exactly one of
$A^1_1=\langle x_1, y_1, h_1 \rangle$,  $A^2_1=\langle x_2, y_2, h_2 \rangle$, or $A^3_1=\langle x_1+x_2, y_1+y_2, h_1+h_2 \rangle$.
We now must determine the ways that we may extend each of these simple subalgebras of $\sll$ by a 
$2$-dimensional subalgebra.

As in the above proof, we may not extend $A^3_1$ by a $2$-dimensional subalgebra since any nontrivial extension of $A^3_1$ is all
of $\sll$.  

The possible extensions of the subalgebra $A^1_1$ by a $2$-dimensional subalgebra are given by Eq. \eqref{acc3}.  Namely
$A^1_1 \oplus K$, where $K \subseteq \{x_2, h_2, y_2 \}$ and is $2$-dimensional.  With respect the adjoint action
of $A^1_1$, $K \cong V(0) \oplus V(0)$ as an $\s$-representation.  Note again that the sum of the extension of $A^1_1$ is direct.

However,  since $K \subseteq \langle x_2,h_2, y_2\rangle$ is solvable and of dimension $2$, up to conjugacy in $SL(2,\mathbb{C})$, we must have $K=\langle x_2, h_2\rangle$.   Hence
\begin{equation}
\begin{array}{lllllll}
A^1_1 \oplus K &\sim& A^1_1 \oplus \langle x_2, h_2\rangle&=&(A_1 \oplus K_2)^1.
\end{array}
\end{equation}
Clearly $\langle x_2, h_2\rangle \cong K_2$.  The equivalency is again given by a matrix of the form
shown in Eq. \eqref{aamhj}.  This follows since $K$ is solvable, $2$ dimensional, and the fact that a Borel subalgebra--also of dimension $2$ for $\s$-- is unique, up to inner automorphism.

In a similar fashion, we have the an extension of $A^2_1$ by a $2$-dimensional subalgebra of $\sll$ is equivalent to 
\begin{equation}
\begin{array}{lllllll}
A^2_1 \oplus \langle x_1, h_1\rangle&=&(A_1 \oplus K_2)^2.
\end{array}
\end{equation}
The subalgebras $(A_1 \oplus K_2)^1$, and $(A_1 \oplus K_2)^2$ are easily seen to be inequivalent.
\end{proof}

\begin{table} [!h]\renewcommand{\arraystretch}{1.6} \caption{Classification of Levi decomposable subalgebras of $\sll$, up to inner automorphism.} \label{enggg}\begin{center}
\begin{tabular}{|c|c|c|clclcl} 
\hline
Dimension &  Levi decomposable subalgebras \\
\hline \hline
4 & $({A_1 \oplus J})^i$, $1\leq i \leq 4$   \\
\hline
5 &$(A_1 \oplus {K_2})^1$, $(A_1 \oplus {K_2})^2$  \\
\hline
\end{tabular}\end{center}
\end{table}

\section{Conclusions}

In this article,  we  
have
classified  the solvable subalgebras, semisimple subalgebras, and Levi
decomposable subalgebras of $\sll$, up to inner automorphism.  By Levi's Theorem, this is a full classification of the subalgebras of $\sll$.  The classification
is summarized in Table \ref{enddd}.

\begin{table}[!h] \renewcommand{\arraystretch}{1.6} \caption{Classification of subalgebras of $\sll$, up to inner automorphism.} \label{enddd}\centerline{
\begin{tabular}{|c|c|c|c|c|c|c|} 
\hline
Dimension &  Semisimple  & Solvable  & Levi decomposable \\
\hline \hline
1& None &$J^1$, $J^2$, $J^3$, $J^4$, $J^5$, $J^6$,  
$J^7$,  $J^{8, a}$,   
& None\\ 
&& 
($a \in \mathbb{C}^{*}$,  
$J^{8,a} \sim J^{8,a'}$ iff $a'  =  \pm a$)
&\\
\hline 
2 & None &  $K_1^1$, $K_1^2$, $K_1^3$, $K_1^4$;   & None\\ 
&&$K^1_2$, $K^{2,a}_2$, $K^{3}_2$, $K^{4,a}_2$, $K^{5}_2$, $a \in \mathbb{C}$&\\
&&
 (for $i=2, 4$:  $K^{i,a}_2 \sim K^{i,a'}_2$ iff $a'  =  \pm a$)
 &\\
\hline 
3 & $A_1^1$, $A_1^2$,  $A_1^3$& $L^1_2$,  $L_{3, a}^{1}$,  $L_{3, a}^{2}$,  $L_{3, 0}^{i}$, $L^1_4$,  &None\\
&& $a\in \mathbb{C}-\{ -\frac{1}{4}, 0 \}$, $1\leq i \leq 4$ & \\
\hline
4 & None & $M_8^1$& $({A_1 \oplus J})^i$, $1\leq i \leq 4$ \\
\hline
5 & None &None& $(A_1 \oplus {K_2})^1$, $(A_1 \oplus {K_2})^2$\\
\hline
\end{tabular}}
\end{table}

\section*{Acknowledgements}

The work of A.D.  is partially supported by a research grant from the Professional Staff Congress/City University of New York (PSC/CUNY).  
The work of J.R. is partially supported by the Natural Sciences and Engineering Research Council (NSERC).

\end{document}